\documentclass[a4paper]{amsart}

\usepackage{amsmath}
\usepackage{amsfonts}
\usepackage{amsthm}
\usepackage{graphicx}
\usepackage{eucal}
\usepackage{amscd}
\usepackage[all,2cell]{xy}
\usepackage{amssymb}
\usepackage{mathrsfs}

 \usepackage{tikz}
 \usetikzlibrary{positioning}
 \tikzset{mynode/.style={draw,circle,inner sep=1pt,outer sep=0pt}}

\newtheorem{teo}{Theorem}[section]
\newtheorem{cor}[teo]{Corollary}
\newtheorem{lem}[teo]{Lemma}
\newtheorem{defi}[teo]{Definition}

\newtheorem{prop}[teo]{Proposition}

\newtheorem{remark}[teo]{Remark}

\newdir{ |>}{{}*!/-3.5pt/@{|}*!/-8pt/:(1,-.2)@^{>}*!/-8pt/:(1,+.2)@_{>}}

\dedicatory{}

\begin{document}

\title{STRICT MONADIC TOPOLOGY II: DESCENT FOR CLOSURE SPACES}

\author{George Janelidze}
\address[George Janelidze]{Department of Mathematics and Applied Mathematics, University of Cape Town, Rondebosch 7700, South Africa}
\thanks{}
\email{george.janelidze@uct.ac.za}

\author{Manuela Sobral}
\address[Manuela Sobral]{CMUC and Departamento de
Matem\'atica, Universidade de Coimbra, 3001--501 Coimbra,
Portugal}
\thanks{
Partially supported by the Centre for Mathematics of
the University of Coimbra -- UID/MAT/00324/2020}
\email{sobral@mat.uc.pt}



\begin{abstract}

By a closure space we will mean a pair $(A,\mathcal{C})$, in which $A$ is a set and $\mathcal{C}$ a set of subsets of $A$ closed under arbitrary intersections. The purpose of this paper is to initiate a development of descent theory of closure spaces, with our main results being: (a) characterization of descent morphisms of closure spaces; (b) in the category of finite closure spaces every descent morphism is an effective descent morphism; (c) every surjective closed map and every surjective open map of closure spaces is an effective descent morphism.

\end{abstract}

\date{\today}

\maketitle

\section{Introduction}

By a \textit{closure space} we will mean a pair $(A,\mathcal{C})$, in which $A$ is a set and $\mathcal{C}$ a set of subsets of $A$ closed under arbitrary intersections; we will also write informally $\mathcal{C}=\mathcal{C}_A$ and $A=(A,\mathcal{C})=(A,\mathcal{C}_A)$. A closure space structure $\mathcal{C}$ on a set $A$ can be equivalently described as a closure operator on the power set $\mathrm{P}(A)$ of $A$ written as $X\mapsto\overline{X}$ (or, more precisely, as $X\mapsto\overline{X}^A$) and satisfying
$$X\subseteq X'\Rightarrow\overline{X}\subseteq\overline{X'},\,\,\,X\subseteq\overline{X},\,\,\,\overline{\overline{X}}=\overline{X}.$$
The relationship between these two types of structures is given by
$$\overline{X}=\bigcap_{X\subseteq A'\in\mathcal{C}}A'\,\,\,\text{and}\,\,\,X\in\mathcal{C}\Leftrightarrow X=\overline{X}.$$

Our reason of using this notion comes from what we called \textit{strict monadic topology} in \cite{[JS2020]}:

Indeed, for a monad $T$ on the category of sets and a $T$-algebra $A$, we can make $A$ a closure space by taking $\mathcal{C}_A$ to be set of all $T$-subalgebras of $A$ -- and then, conversely, every closure space is of this form for a suitably chosen monad.

The purpose of this paper is to initiate a development of descent theory of closure spaces, specifically to:
\begin{itemize}
	\item characterize descent morphisms (=pullback stable regular epimorphisms) of closure spaces (Proposition 2.10);
	\item prove that in the category of finite closure spaces every descent morphism is an effective descent morphism (Theorem 4.3);
	\item compare the above-mentioned result with what happens with finite topological spaces;
	\item prove that surjective closed maps and surjective open maps of closure spaces are always effective descent morphisms (Theorem 6.5).
\end{itemize}

The paper is organized as follows: we begin with (mostly known, maybe in slightly different contexts) auxiliary results on closure spaces in Section 2 and on general descent theory in Section 3, except that Section 2 also includes the above-mentioned Proposition 2.10; Sections 4-6 are devoted to other main results, and Section 7 to some additional remarks and open questions.

\section{Closure spaces}

We will consider the category $\mathbf{CLS}$ of closure spaces, where a morphism $\alpha:A\to B$ is a map $\alpha$ from $A$ to $B$ with
$$B'\in\mathcal{C}_B\Rightarrow\alpha^{-1}(B')\in\mathcal{C}_A.$$

It is easy to see that the underlying set functor $U:\mathbf{CLS}\to\mathbf{Sets}$ is \textit{topological} in the sense of categorical topology, which then easily gives the Propositions 2.1 and 2.2 below:

\begin{prop}	A diagram in $\mathbf{CLS}$ of the form
	$$\xymatrix{D\ar[d]_-{\pi_1}\ar[r]^{\pi_2}&A\ar[d]^\alpha\\E\ar[r]_-p&B}$$
	is a pullback diagram in $\mathbf{CLS}$ if and only if its $U$-image is a pullback diagram in $\mathbf{Sets}$ and $\mathcal{C}_D=\{\pi_1^{-1}(E')\cap\pi_2^{-1}(A')\mid E'\in\mathcal{C}_E\,\&\,A'\in\mathcal{C}_A\}$.\qed
\end{prop}
We will, however, present the diagram above as
$$\xymatrix{E\times_BA\ar[d]_{\pi_1}\ar[r]^-{\pi_2}&A\ar[d]^\alpha\\E\ar[r]_-p&B}$$
informally identifying $E\times_BA$ with $\{(e,a)\in E\times A\mid p(e)=\alpha(a)\}$, and write
$$\mathcal{C}_{E\times_BA}=\{E'\times_BA'=\pi_1^{-1}(E)\cap\pi_2^{-1}(A)\mid E'\in\mathcal{C}_E\,\&\,A'\in\mathcal{C}_A\}.$$ \textit{We will refer to this diagram as the pullback diagram for} $(p,\alpha)$.
\begin{prop}	A diagram in $\mathbf{CLS}$ of the form
	$$\xymatrix{F\ar@<0.5ex>[r]^-{p_1}\ar@<-0.5ex>[r]_-{p_2}&E\ar[r]^p&B}$$
	is a coequalizer diagram in $\mathbf{CLS}$ if and only if its $U$-image is a coequalizer diagram in $\mathbf{Sets}$ and $\mathcal{C}_B=\{B'\subseteq B\mid p^{-1}(B')\in\mathcal{C}_E\}$.\qed
\end{prop}
\begin{cor}
	A morphism $p:E\to B$ in $\mathbf{CLS}$ is a regular epimorphism if and only if $p$ is a surjective map with $\mathcal{C}_B=\{B'\subseteq B\mid p^{-1}(B')\in\mathcal{C}_E\}$.\qed
\end{cor}

Most of what we present in the rest of this section either automatically extends what is known for topological spaces, or known itself, possibly as `folklore', or is presented in some form in \cite{[M2011]}:
\begin{prop}
	For closure spaces $E$ and $B$, and a map $p:E\to B$, the following conditions are equivalent:
	\begin{itemize}
		\item [(a)] $p:E\to B$ is a morphism in $\mathbf{CLS}$;
		\item [(b)] $\overline{p^{-1}(X)}\subseteq p^{-1}(\overline{X})$ for every $X\subseteq B$;
		\item [(c)] $p(\overline{p^{-1}(X)})\subseteq\overline{X}$ for every $X\subseteq B$;
		\item [(d)] $p(\overline{Y})\subseteq\overline{p(Y)}$ for every $Y\subseteq E$;
		\item [(e)] $\overline{Y}\subseteq p^{-1}(\overline{p(Y)})$ for every $Y\subseteq E$.
	\end{itemize}
\end{prop}
\begin{proof}
	(a)$\Rightarrow$(b): Since $X\subseteq\overline{X}$, we have $p^{-1}(X)\subseteq p^{-1}(\overline{X})$, and then $\overline{p^{-1}(X)}\subseteq\overline{p^{-1}(\overline{X})}$, but $\overline{p^{-1}(\overline{X})}=p^{-1}(\overline{X})$ by (a), since $\overline{X}\in\mathcal{C}_B$.
	
	(b)$\Rightarrow$(a): If $B'\in\mathcal{C}_B$, then $\overline{B'}=B'$ and (b) gives $\overline{p^{-1}(B')}\subseteq p^{-1}(B')$, making $\overline{p^{-1}(B')}=p^{-1}(B')$ and so making $\overline{p^{-1}(B')}\in\mathcal{C}_E$.
	
	(b)$\Leftrightarrow$(c) and (d)$\Leftrightarrow$(e) are obvious.
	
	(b)$\Rightarrow$(e): Since $Y\subseteq p^{-1}(p(Y))$, we have $\overline{Y}\subseteq\overline{p^{-1}(p(Y))}$, but $\overline{p^{-1}(p(Y))}\subseteq p^{-1}(\overline{p(Y)})$ by (b).
	
	(d)$\Rightarrow$(c): Since $p(p^{-1}(X))\subseteq X$, we have $\overline{p(p^{-1}(X))}\subseteq\overline{X}$, but $p(\overline{p^{-1}(X)})\subseteq\overline{p(p^{-1}(X))}$ by (d).
\end{proof}
\begin{prop}
	The following conditions on a morphism $p:E\to B$ in $\mathbf{CLS}$ are equivalent:
	\begin{itemize}
		\item [(a)] $p$ is closed, that is, $Y\in\mathcal{C}_E\Rightarrow p(Y)\in\mathcal{C}_B$;
		\item [(b)] $p(\overline{Y})\supseteq\overline{p(Y)}$ for every $Y\subseteq E$;
		\item [(c)] $p(\overline{Y})=\overline{p(Y)}$ for every $Y\subseteq E$.
	\end{itemize}
\end{prop}
\begin{proof}
	(a)$\Rightarrow$(b): Since $Y\subseteq\overline{Y}$, we have $p(Y)\subseteq p(\overline{Y})$ and then $\overline{p(Y)}\subseteq\overline{p(\overline{Y})}$, but $\overline{p(\overline{Y})}=p(\overline{Y})$ by (a), since $\overline{Y}\in\mathcal{C}_E$.
	
	(b)$\Rightarrow$(c) follows from the implication (a)$\Rightarrow$(d) of Proposition 2.4.
	
	$(c)\Rightarrow$(a): If $Y=\overline{Y}$, then $p(Y)=\overline{p(Y)}$ by (c).
\end{proof}
\begin{prop}
	The following conditions on a morphism $p:E\to B$ in $\mathbf{CLS}$ are equivalent:
	\begin{itemize}
		\item [(a)] $p$ is open, that is, $-Y\in\mathcal{C}_E\Rightarrow-p(Y)\in\mathcal{C}_B$;
		\item [(b)] $\overline{X}\subseteq-p(-\overline{p^{-1}(X)})$ for every $X\subseteq B$;
		\item [(c)] $\overline{p^{-1}(X)}\supseteq p^{-1}(\overline{X})$ for every $X\subseteq B$;
		\item [(d)] $\overline{p^{-1}(X)}=p^{-1}(\overline{X})$ for every $X\subseteq B$.
	\end{itemize}
\end{prop}
\begin{proof}
	(a)$\Rightarrow$(b): Since $\overline{p^{-1}(X)}\in\mathcal{C}_E$, we have $-p(-\overline{p^{-1}(X)})\in\mathcal{C}_B$ by (a). Therefore to deduce (c) is to show that $X\subseteq-p(-\overline{p^{-1}(X)})$, but we have
	$$X\subseteq-p(-\overline{p^{-1}(X)})\Leftrightarrow p(-\overline{p^{-1}(X)})\subseteq-X\Leftrightarrow-\overline{p^{-1}(X)}\subseteq p^{-1}(-X)$$
	$$\Leftrightarrow-\overline{p^{-1}(X)}\subseteq-p^{-1}(X)\Leftrightarrow p^{-1}(X)\subseteq\overline{p^{-1}(X)}.$$
	
	(b)$\Rightarrow$(a): Applying (b) to $X=-p(Y)$, we obtain the first inclusion in
	$$\overline{-p(Y)}\subseteq-p(-\overline{p^{-1}(-p(Y))})=-p(-\overline{-p^{-1}(p(Y))})\subseteq-p(-\overline{-Y}),$$
	and for $-Y\in\mathcal{C}_E$ this gives $\overline{-p(Y)}\subseteq-p(--Y)=-p(Y)$, which means that $-p(Y)\in\mathcal{C}_B$.
	
	(b)$\Leftrightarrow$(c): We have
	$$\overline{X}\subseteq-p(-\overline{p^{-1}(X)})\Leftrightarrow p(-\overline{p^{-1}(X)})\subseteq-\overline{X}\Leftrightarrow-\overline{p^{-1}(X)}\subseteq p^{-1}(-\overline{X})$$
	$$\Leftrightarrow-\overline{p^{-1}(X)}\subseteq-p^{-1}(\overline{X})\Leftrightarrow \overline{p^{-1}(X)}\supseteq p^{-1}(\overline{X}).$$
	
	(c)$\Leftrightarrow$(d) follows from the implication (a)$\Rightarrow$(b) of Proposition 2.4.
\end{proof}

For a morphism $p:E\to B$ in $\mathbf{CLS}$ and $X\subseteq B$, let us define $p_\infty(X)$ by transfinite induction as follows:
$$p_0(X)=X,\,\,\,p_{\lambda+1}(X)=p(\overline{p^{-1}(p_{\lambda}(X))})=p_1(p_\lambda(X)),$$
$$p_{\mu}(X)=\bigcup_{\lambda<\mu}p_{\lambda}(X)\,\text{(for a limit ordinal}\,\mu),\,\,\,p_\infty(X)=\bigcup_{\lambda}p_{\lambda}(X).$$
Note that, $p_1(X)\subseteq p(E)$ and, using transfinite induction, we conclude that $p_\infty(X)\subseteq p(E)$
for every $X\subseteq B$. Furthermore, when $p$ is surjective, we have $X\subseteq p_1(X)$, and so
$$\lambda\leqslant\mu\Rightarrow p_\lambda(X)\subseteq p_\mu(X)\,(\subseteq p_\infty(X)).$$
\begin{prop}
	The following conditions on a morphism $p:E\to B$ in $\mathbf{CLS}$ are equivalent:
	\begin{itemize}
		\item [(a)] $p$ is a regular epimorphism;
		\item [(b)] $\overline{X}\subseteq p_\infty(X)$ for every $X\subseteq B$;
		\item [(c)] $\overline{X}=p_\infty(X)$ for every $X\subseteq B$.
	\end{itemize}
\end{prop}
\begin{proof}
	(a)$\Rightarrow$(c): Suppose $p$ is a regular epimorphism, and so
	$$\mathcal{C}_B=\{B'\subseteq B\mid p^{-1}(B')\in\mathcal{C}_E\}$$
	by Corollary 2.3. Let $C$ be the closure space whose underlying set is the same as for $B$ and whose closure operator is $p_\infty$, that is, it is defined by $\overline{X}=p_\infty(X)$ (all required conditions for a closure operator are obviously satisfied). We have
	$$X\in\mathcal{C}_C\Leftrightarrow X=p_\infty(X)\Leftrightarrow X=p_1(X)\Leftrightarrow X=p(\overline{p^{-1}(X)})\Leftrightarrow p(\overline{p^{-1}(X)})\subseteq X$$
	$$\Leftrightarrow\overline{p^{-1}(X)}\subseteq p^{-1}(X)\Leftrightarrow\overline{p^{-1}(X)}=p^{-1}(X)\Leftrightarrow p^{-1}(X)\in\mathcal{C}_E,$$
	which means that $C=B$ as closure spaces. That is, (c) holds.
	
	(c)$\Rightarrow$(b) is trivial.
	
	(b)$\Rightarrow$(a): Suppose $\overline{X}\subseteq p_\infty(X)$ for every $X\subseteq B$. First of all we have
	$$B=\overline{B}\subseteq p_\infty(B)\subseteq p(E),$$
	and so $p$ surjective. Next, take any $X\subset B$ with $p^{-1}(X)\in\mathcal{C}_E$; we have
	$$p_1(X)=p(\overline{p^{-1}(X)})=p(p^{-1}(X))=X,$$
	and then $p_\infty(X)=X$ by transfinite induction. Hence $\overline{X}\subset X$. That is, $X\in\mathcal{C}_B$ whenever $p^{-1}(X)\in\mathcal{C}_E$ and $p$ is a regular epimorphism by Corollary 2.3.
\end{proof}
Consider again the pullback diagram for $(p,\alpha)$:
\begin{prop}
	For $Z\subseteq E\times_BA$ one has $\overline{Z}=\pi_1^{-1}(\overline{\pi_1(Z)})\cap\pi_2^{-1}(\overline{\pi_2(Z)})$
\end{prop}
\begin{proof} Since $\pi_1^{-1}(\overline{\pi_1(Z)})\cap\pi_2^{-1}(\overline{\pi_2(Z)})\in\mathcal{C}_{E\times_BA}$, we only need to prove that if $Z\subseteq\pi_1^{-1}(E')\cap\pi_2^{-1}(A')$ for $E'\in\mathcal{C}_E$ and $A'\in\mathcal{C}_A$, then
$$\pi_1^{-1}(\overline{\pi_1(Z)})\cap\pi_2^{-1}(\overline{\pi_2(Z)})\subseteq\pi_1^{-1}(E')\cap\pi_2^{-1}(A').$$
We have
$$Z\subseteq\pi_1^{-1}(E')\cap\pi_2^{-1}(A')\Rightarrow Z\subseteq\pi_1^{-1}(E')\Rightarrow\pi_1(Z)\subseteq E'\Rightarrow\overline{\pi_1(Z)}\subseteq E',$$
where the last implication follows from $E'\in\mathcal{C}_E$. That is, we can write $\overline{\pi_1(Z)}\subseteq E'$; similarly $\overline{\pi_2(Z)}\subseteq A'$. Now, for $(e,a)\in\pi_1^{-1}(\overline{\pi_1(Z)})\cap\pi_2^{-1}(\overline{\pi_2(Z)})$, we have
$$e=\pi_1(e,a)\in\overline{\pi_1(Z)}\subseteq E'\,\,\,\text{and}\,\,\,a=\pi_2(e,a)\in\overline{\pi_2(Z)}\subseteq A',$$
and so $(e,a)\in\pi_1^{-1}(E')\cap\pi_2^{-1}(A')$, as desired.	
\end{proof}

Let $S$ be a subset of (the underlying of) a closure space $B$, and $\iota:S\to B$ the inclusion map. This makes $S$ a closure space, which we will denote by $S_B$, and which has
$$\mathcal{C}_{S_B}=\{S\cap B'\mid B'\in C_B\}\,\,\,\text{and}\,\,\,\overline{U}^{S_B}=S\cap\overline{U}^B$$
for every $U\subseteq S$. From Proposition 2.8, or directly, we easily obtain
\begin{prop}
	For a morphism $p:E\to B$ in $\mathbf{CLS}$ and a subset $S$ of $B$, the diagram
	$$\xymatrix{p^{-1}(S)_E\ar[d]_-{\kappa}\ar[r]^-{p'}&S_B\ar[d]^\iota\\E\ar[r]_-p&B},$$
	in which $\iota$ and $\kappa$ are the inclusion maps, is a pullback diagram in $\mathbf{CLS}$.\qed
\end{prop}
\begin{prop}
	The following conditions on a morphism $p:E\to B$ in $\mathbf{CLS}$ are equivalent:
	\begin{itemize}
		\item [(a)] $p$ is a pullback stable regular epimorphism;
		\item [(b)] $\overline{X}\subseteq p(\overline{p^{-1}(X)})$ for every $X\subseteq B$;
		\item [(c)] $\overline{X}=p(\overline{p^{-1}(X)})$ for every $X\subseteq B$;
		\item [(d)] $p(\overline{p^{-1}(X)})$ is closed for every $X\subseteq B$
	\end{itemize}
\end{prop}
\begin{proof}
	(a)$\Rightarrow$(b): Given $X\subseteq B$, consider the pullback diagram of Proposition 2.9 with $X\subseteq S\subseteq B$. Assuming (a), $p'$ must be a regular epimorphism, and, in particular,
	$$S\cap\overline{X}^B=\overline{X}^{S_B}=p'_{\infty}(X),$$
	where the second equality follows from the implication (a)$\Rightarrow$(c) of Proposition 2.7.
	We take $S=X\cup-p(Y)$ with $Y=\overline{p^{-1}(X)}^E$ and calculate:
	$$p'_1(X)=p'(\overline{p'^{-1}(X)}^{p^{-1}(S)_E})=p(\overline{p^{-1}(X)}^{p^{-1}(S)_E})=p(p^{-1}(S)\cap\overline{p^{-1}(X)}^E)$$
	$$=p(p^{-1}(S)\cap Y)=p(p^{-1}(X\cup-p(Y))\cap Y)$$
	$$=p((p^{-1}(X)\cap Y)\cup(-p^{-1}(p(Y))\cap Y))=X,$$
	where the last equality follows from $p^{-1}(X)\subseteq\overline{p^{-1}(X)}^E=Y$, $Y\subseteq p^{-1}(p(Y))$, and $p(p^{-1}(X))=X$. Since $p'_1(X)=X$, using transfinite induction we also obtain $p'_\infty(X)=X$. This gives
	$$(X\cup-p(\overline{p^{-1}(X)}^E))\cap\overline{X}^B=S\cap\overline{X}^B=p'_{\infty}(X)=X,$$
	which implies $-p(\overline{p^{-1}(X)}^E)\cap\overline{X}^B\subseteq X$. Since $X=p(p^{-1}(X))\subseteq p(\overline{p^{-1}(X)}^E)$, it follows that $\overline{X}^{B}\subseteq p(\overline{p^{-1}(X)}^{E})$, as desired.
	
	(b)$\Leftrightarrow$(c) follows from the implication (a)$\Rightarrow$(c) of Proposition 2.4, and (c)$\Leftrightarrow$(d) follows from Proposition 2.7.
	
	(c)$\Rightarrow$(a): Suppose (c) holds. We have to prove that, for every pullback diagram as in Proposition 2.8, $\pi_2$ is a regular epimorphism. Thanks to the implication (b)$\Rightarrow$(a) of Proposition 2.7, it suffices to prove that
	$$\overline{U}\subseteq\pi_2(\overline{\pi_2^{-1}(U)})$$
	for every $U\subseteq A$. We have:
	$$\overline{U}\subseteq\alpha^{-1}(\overline{\alpha(U)})\cap\overline{U}\,\,\,(\text{since}\,\,\alpha(\overline{U})\subseteq\overline{\alpha(U)})\,\,\text{gives}\,\,\overline{U}\subseteq\alpha^{-1}(\overline{\alpha(U)})$$
	$$=\alpha^{-1}(p(\overline{p^{-1}(\alpha(U))}))\cap\overline{U}\,\,\,(\text{by (c)})$$
	$$=\pi_2(\pi_1^{-1}(\overline{\pi_1(\pi_2^{-1}(U))}))\cap\overline{U}\,\,\,(\text{Beck--Chevalley Condition used twice})$$
	$$=\pi_2(\pi_1^{-1}(\overline{\pi_1(\pi_2^{-1}(U)}))\cap\pi_2^{-1}(\overline{U}))\,\,\,(\text{another Beck--Chevalley Condition})$$
	$$=\pi_2(\pi_1^{-1}(\overline{\pi_1(\pi_2^{-1}(U))})\cap\pi_2^{-1}(\overline{\pi_2(\pi_2^{-1}(U))}))=\pi_2(\overline{\pi_2^{-1}(U)}),$$
	as desired.
\end{proof}

\section{General remarks on descent}

In this section $\mathbf{C}$ denotes a category with pullbacks and coequalizers of equivalence relations. All pullback projections will denoted by $\pi$'s with suitable indices.

We will list notions and results of general descent theory in the form convenient for our purposes, not repeating any motivations and further explanations that can be found in \cite{[JT1994]} or in \cite{[JST2004]}; we will also use a particular result from \cite{[ST1992]}.

\begin{defi}
	\em{Let $p:E\to B$ be a morphism in $\mathbf{C}$. Then:}
	\begin{itemize}
		\item [(a)] A \textit{descent data} for $p$ is a triple $(C,\gamma,\xi)$ as in the diagram
		$$\xymatrix{E\times_B(E\times_BC)\ar[d]_{E\times_B\pi_2}\ar[r]^-{E\times_B\xi}&E\times_BC\ar[d]_{\xi}&C\ar[l]_-{\langle\gamma,1_C\rangle}\ar@{=}[dl]\\E\times_BC\ar[d]_{\pi_1}\ar[r]^-{\xi}&C\ar[dl]^{\gamma}\\E}$$
		(in obvious notation), which is required to commute. The category of all such triples will be denoted by $\mathrm{Des}(p)$.
		\item [(b)] The functor $K^p:(\mathbf{C}\downarrow B)\to\mathrm{Des}(p)$, defined by
		$$\xymatrix{K^p(A,\alpha)=(E\times_B(E\times_BA)\ar[r]^-{E\times_B\pi_2}&E\times_BA\ar[r]^-{\pi_1}&E)}$$
		is called the \textit{comparison functor} (for $p$).
		\item [(c)] The morphism $p$ is said to be a \textit{descent morphism} if the functor $K^p$ is fully faithful.
		\item [(d)] The morphism $p$ is said to be an \textit{effective $\mathsf{E}$-descent morphism} if the functor $K^p$ is a category equivalence.
	\end{itemize}
\end{defi}

\begin{remark}
	\em{Each of the following statements is either well known or immediately follows from well-known facts:}
	\begin{itemize}
		\item [(a)] If $(C,\gamma,\xi)$ is a descent data for $p:E\to B$, then
		$$\xymatrix{E\times_BC\ar[d]_{E\times_B\gamma}\ar@<0.5ex>[r]^-{\xi}\ar@<-0.5ex>[r]_-{\pi_2}&C\ar[d]^{\gamma}\\E\times_BE\ar@<0.5ex>[r]^-{\pi_1}\ar@<-0.5ex>[r]_-{\pi_2}&E}$$
		is a discrete fibration of equivalence relations. Moreover, sending $(C,\gamma,\xi)$ to this discrete fibration determines a category equivalence
		$$\mathrm{Des}(p)\to\mathrm{DFib}(\mathrm{Eq}(p)),$$
		where $\mathrm{DFib}(\mathrm{Eq}(p))$ is the category of discrete fibrations of equivalence relations whose codomain is
		$$\xymatrix{\mathrm{Eq}(p)=(E\times_BE\ar@<0.5ex>[r]^-{\pi_1}\ar@<-0.5ex>[r]_-{\pi_2}&E).}$$
		\item [(b)] Suppose $p$ is a regular epimorphism, and so we can assume that $B$ (equipped with $p$) is the coequalizer of the bottom equivalence relation in (a). Then sending $(C,\gamma,\xi)$ to the morphism of the coequalizers of equivalence relations in (a) determines a left adjoint $L^p$ of $K^p$.
		\item [(c)] As follows from (a) and (b), $p$ is an effective descent morphism if and only if it is a descent morphism and the functor $L^p$ reflects isomorphisms, or, equivalently, the coequalizer functor
		$$\mathrm{DFib}(\mathrm{Eq}(p))\to(\mathbf{C}\downarrow B)$$
		does so.
		\item [(d)] A morphism in $\mathbf{C}$ is a descent morphism if and only if it is a pullback stable regular epimorphism.
		\item [(e)] As easily follows from previous observations, every descent morphism in $\mathbf{C}$ is an effective descent morphism if and only if for every descent morphism $p:E\to B$ and every diagram of the form
		$$\xymatrix{E\times_BC\ar[d]_{E\times_B\gamma}\ar@<0.5ex>[r]^-{\xi}\ar@<-0.5ex>[r]_-{\pi_2}&C\ar[d]^{\gamma}\ar[r]^q&A\ar[d]^\alpha\\E\times_BE\ar@<0.5ex>[r]^-{\pi_1}\ar@<-0.5ex>[r]_-{\pi_2}&E\ar[r]_p&B}$$
		where $(C,\gamma,\xi)$ is a descent data for $p$, the top row is a coequalizer diagram, the right-hand square commutes, and $\alpha$ is an isomorphism, $\gamma$ also is an isomorphism. More generally, if $\mathbf{D}$ is a pullback stable class of morphisms containing the class of descent morphisms and satisfying the condition above (with $p$ is in $\mathbf{D}$), then $\mathbf{D}$ is contained in the class of effective descent morphisms.
		\item [(f)] A regular epimorphism $p$ in $\mathbf{CLS}$ is an effective descent morphism if and only if, for each descent data $(C,\gamma,\xi)$ for $p$, the coequalizer of $$\xymatrix{E\times_BC\ar@<0.5ex>[r]^-{\xi}\ar@<-0.5ex>[r]_-{\pi_2}&C}$$
		is a pullback stable regular epimorphism. This follows from the observation in \cite{[ST1992]} made immediately after Corollary 2.8 there.
	\end{itemize}
\end{remark}

\section{Descent for closure spaces}

In this section we go back to the category $\mathbf{CLS}$ of closure spaces and $p:E\to B$ will denote a fixed morphism there, which is a surjective map. We will also use a closure space $E'$, which has the same underlying set as $E$, and, for $X\subseteq B$, $Y\subseteq E$, and $Z\subseteq E\times_BE'$, write
$$\overline{X}=\overline{X}^B,\,\,\,\overline{Y}=\overline{Y}^E,\,\,\,\overline{Y}'=\overline{Y}^{E'},\,\,\,\text{and}\,\,\,\overline{Z}=\overline{Z}^{E\times_BE'}.$$

\begin{lem}
	Suppose the identity map $1_E:E'\to E$ is a morphism in $\mathbf{CLS}$, that is, $\overline{Y}'\subseteq\overline{Y}$ for all $Y\subseteq E$. Then the following conditions are equivalent:
	\begin{itemize}
		\item [(a)] there exists a descent data for $p$ of the form $(E',1_E,\xi)$;
		\item [(b)] there exists a unique descent data for $p$ of the form $(E',1_E,\xi)$;
		\item [(c)] the triple $(E',1_E,\pi_1)$, where $\pi_1:E\times_BE'\to E'$ is defined as the first projection, that is, by $\pi_1(e,e')=e$, is a descent data for $p$;
		\item [(d)] the first projection $\pi_1:E\times_BE'\to E'$ is a morphism in $\mathbf{CLS}$;
		\item [(e)] $\overline{Y}\cap p^{-1}(p(\overline{p^{-1}(p(Y))}'))\subseteq\overline{Y}'$ for all $Y\subseteq E$;
		\item [(f)] $\overline{Y}\cap p^{-1}(p(\overline{p^{-1}(p(Y))}'))=\overline{Y}'$ for all $Y\subseteq E$;
		\item [(g)] $\overline{Y}\cap p^{-1}(p(\overline{p^{-1}(p(Y))}'))\subseteq Y$ for all $Y\in\mathcal{C}_{E'}$;
		\item [(h)] $\overline{Y}\cap p^{-1}(p(\overline{p^{-1}(p(Y))}'))= Y$ for all $Y\in\mathcal{C}_{E'}$.
	\end{itemize}
\end{lem}
\begin{proof}
	The implications (a)$\Leftrightarrow$(b)$\Rightarrow$(c)$\Rightarrow$(d) follow from the commutativity of the bottom triangle of the diagram in 3.2(a), where $\gamma$ becomes the map $1_E:E'\to E$ in the present case. The implication (d)$\Rightarrow$(c) can be checked with a straightforward calculation and the implication (c)$\Rightarrow$(a) is trivial. Hence conditions (a)-(d) are all equivalent to each other.
	
	(d)$\Leftrightarrow$(e): As follows from the equivalence (a)$\Leftrightarrow$(c) of Proposition 2.4, condition (d) holds if and only if
	$$\pi_1(\overline{\pi_1^{-1}(Y)})\subseteq\overline{Y}'$$
	for all $Y\subseteq E$. Using Proposition 2.8, we obtain:
	$$\pi_1(\overline{\pi_1^{-1}(Y)})=\pi_1(\overline{Y\times_BE})=\pi_1(\pi_1^{-1}(\overline{\pi_1(Y\times_BE)})\cap\pi_2^{-1}(\overline{\pi_2(Y\times_BE)}'))$$
	$$=\overline{\pi_1(Y\times_BE)}\cap\pi_1(\pi_2^{-1}(\overline{\pi_2(Y\times_BE)}'))=\overline{Y}\cap\pi_1(\pi_2^{-1}(\overline{p^{-1}(p(Y))}'))$$
	$$=\overline{Y}\cap p^{-1}(p(\overline{p^{-1}(p(Y))}')),$$
	and so (d) is indeed equivalent to (e).
	
	Since $\overline{Y}'\subseteq\overline{Y}$ and $\overline{Y}'\subseteq p^{-1}(p(\overline{p^{-1}(p(Y))}'))$, we have (e)$\Leftrightarrow$(f); similarly, we have (g)$\Leftrightarrow$(h). (e)$\Leftrightarrow$(g) is also straightforward.
\end{proof}

\begin{lem}
	Suppose the equivalent conditions of Lemma 4.1 are satisfied and let us write $p'$ for $p$ considered as a morphism from $E'$ to $B$. If both $p$ and $p'$ are regular epimorphisms, then, for every $Y\in\mathcal{C}_{E'}\setminus\mathcal{C}_E$, there exists $Y^*\in\mathcal{C}_{E'}\setminus\mathcal{C}_E$ with $Y\subset Y^*$. In particular, if $\mathcal{C}_{E'}\neq\mathcal{C}_E$, then $E$ is infinite.
\end{lem}
\begin{proof}
	For $Y\in\mathcal{C}_{E'}\setminus\mathcal{C}_E$, we have $Y\subset p^{-1}(p(Y))$. Indeed, since $p$ and $p'$ are regular epimorphisms, the equality $Y=p^{-1}(p(Y))$ would imply
	$$Y\in\mathcal{C}_{E'}\Leftrightarrow p(Y)\in\mathcal{C}_B\Leftrightarrow Y\in\mathcal{C}_E$$
	(by Corollary 2.3), which is a contradiction.
	
	Let us take
	$$Y^*=\overline{p^{-1}(p(Y))}'.$$
	We have $Y\subset Y^*$ and $Y^*\in\mathcal{C}_{E'}$. Therefore it remains to show that $Y^*$ does not belong to $\mathcal{C}_{E}$. Suppose it does. Then, since it contains $Y$ as a subset, we have $\overline{Y}\subseteq Y^*$. This gives
	$$\overline{Y}=\overline{Y}\cap Y^*=\overline{Y}\cap\overline{p^{-1}(p(Y))}'\subseteq\overline{Y}\cap p^{-1}(p(\overline{p^{-1}(p(Y))}'))=Y$$
	(the last equality here is condition (h) of Lemma 4.1), which is a contradiction since $Y$ does not belong to $\mathcal{C}_{E}$.
\end{proof}

Let $\mathbf{FCLS}$ be the category of finite closure spaces, that is, the full subcategory of $\mathbf{FCLS}$ with objects all closure spaces whose underlying sets are finite. From Remark 3.3(e) and Lemma 4.2 we obtain:

\begin{teo}
	Every descent morphism in the category $\mathbf{FCLS}$ is an effective descent morphism.\qed
\end{teo}	

\section{Preorders as closure spaces}

There are full inclusions
$$\mathbf{Preord}\to\mathbf{Top}\to\mathbf{CLS},$$
where $\mathbf{Preord}$ is the category of preorders (=preordered sets) and $\mathbf{Top}$ is the category of topological spaces. Considering a preorder $B$ as either a topological space or a closure space, for any $X\subseteq B$, we have
$$\overline{X}=\,\,\uparrow\!X=\{b\in B\mid\exists_{x\in X}\,x\leqslant b\}.$$
As mentioned in Remark 2.4(b) of \cite{[CJ2020]}, not every descent morphism in $\mathbf{Preord}$ is a descent morphism in $\mathbf{Top}$; nevertheless we have:
\begin{prop}
	A morphism in $\mathbf{Preord}$ is a descent morphism in $\mathbf{Preord}$ if and only if it is a descent morphism in $\mathbf{CLS}$.
\end{prop}
\begin{proof}
	Let $p:E\to B$ be a morphism in $\mathbf{Preord}$. As shown in \cite{[JS2002]}, $p$ is a descent morphism in $\mathbf{Preord}$ if and only if for all $b\leqslant b'$ in $B$ there exist $e\leqslant e'$ in $E$ with $p(e)=b$ and $p(e')=b'$. This, in turn, is easily equivalent to
	$$p(\overline{p^{-1}(X)})=p(\uparrow\!p^{-1}(X))\supseteq\,\uparrow\!X=\overline{X},$$
	and it remain to apply Proposition 2.10 and Remark 3.2(d).
\end{proof}

On the other hand, the result similar to Theorem 4.3 does not hold in $\mathbf{Preord}$, and not even in the category $\mathbf{FPreord}$ of finite preorders \cite{[JS2002]}. In order to clarify the phenomenon behind this, consider the following example, the simplest one in a sense:

Let $p:E\to B$ be the morphism in $\mathbf{FPreord}$, and $\alpha:A\to B$ to be the morphism in the category $\mathbf{FRR}$ of finite reflexive relations (=sets equipped with a reflexive relation) definied as follows:
\begin{itemize}
	\item $B=\{b_1,b_2,b_3\}$ is the ordered set with $b_1<b_2<b_3$.
	\item $E=\{e_1,e_{2-},e_{2+},e_3\}$ is the ordered set with $e_1<e_{2-}$, $e_{2+}<e_3$, $e_1<e_3$, and no other strict inequalities.
	\item $p(e_1)=b_1$, $p(e_{2-})=b_2=p(e_{2+})$, and $p(e_3)=b_3$.
	\item $A=B$ but with the pair $(b_1,b_3)$ removed from the relation.
	\item $\alpha$ is the identity map of $B$ considered as a morphism from $A$ to $B$.
\end{itemize}
The pullback $E'=E\times_BA$ of $p$ and $\alpha$ can be identified with the ordered set $E=\{e_1,e_{2-},e_{2+},e_3\}$ with $e_1<e_{2-}$, $e_{2+}<e_3$, and no other strict inequalities. And after that the pullback $E\times_BE'$ can be presented as the diagram
$$\xymatrix{&(e_1,e_1)\ar[dl]\\(e_{2-},e_{2-})&(e_{2-},e_{2+})\,\,\,\,\,\,\,\,\,\,\,\,\,\,\,\,(e_{2+},e_{2-})&(e_{2+},e_{2+})\ar[dl]\\&(e_3,e_3)}$$
whose vertexes are its elements and whose arrows represent strict inequalities. We observe:
\begin{itemize}
	\item [(a)] Although $A$ is not a preorder, $E'$ is. This tells us that $(E',1_E,\pi_1)$ is a descent data for $p$ in $\mathbf{FPreord}$. Comparing it with $(E,1_E,\pi_1)$ is a simple way to show that $p$ is not an effective descent morphism in $\mathbf{FPreord}$.
	\item [(b)] The set $Y=\{e_1,e_{2,-}\}$ is closed in $E'$ and its inverse image
	$$Z=\pi_1^{-1}(Y)=\{(e_1,e_1),(e_{2-},e_{2-}),(e_{2-},e_{2+})\}$$
	is closed of course in the pullback $E\times_BE'$ displayed above.
	\item [(c)] However, if we define $E\times_BE'$ as the pullback in $\mathbf{FCLS}$, then
	$$\overline{Z}=\pi_1^{-1}(\overline{\pi_1(Z)})\cap\pi_2^{-1}(\overline{\pi_2(Z)}')=\pi_1^{-1}(\overline{\{e_1,e_{2-}\}})\cap\pi_2^{-1}(\overline{\{e_1,e_{2-},e_{2+}\}}')$$
	$$=\pi_1^{-1}(\{e_1,e_{2-},e_3\})\cap\pi_2^{-1}(\{e_1,e_{2-},e_{2+},e_3\})=\pi_1^{-1}(\{e_1,e_{2-},e_3\})$$
	$$=\{(e_1,e_1),(e_{2-},e_{2-}),(e_{2-},e_{2+}),(e_3,e_3)\}\neq Z,$$
	and so $Z$ will not be closed anymore.
	\item [(d)] As follows from (c), for the pullback $E\times_BE'$ defined as in $\mathbf{FCLS}$, the map $\pi_1:E\times_BE'\to E'$ is not a morphism in $\mathbf{FCLS}$. Therefore there is no `bad' descent data $(E',1_E,\pi_1)$ in $\mathbf{FCLS}$, to prevent $p$ from being an effective descent morphism.
\end{itemize}
Of course this is only an example of one preorder argument that does not hold for closure spaces and it cannot replace the proof of Theorem 4.3, but it shows a crucial difference between the descent stories of preorders and of closure spaces.

Furthermore, in the pullback $E\times_BE'$ defined as in $\mathbf{FCLS}$, putting $Z=U\cup V$ with
$U=\{(e_1,e_1),(e_{2-},e_{2-})\}$ and $V=\{(e_{2-},e_{2+})\}$, we calculate
$$\overline{U}=\pi_1^{-1}(\overline{\pi_1(U)})\cap\pi_2^{-1}(\overline{\pi_2(U)}')=\pi_1^{-1}(\overline{\{e_1,e_{2-}\}})\cap\pi_2^{-1}(\overline{\{e_1,e_{2-}\}}')$$
$$=\pi_1^{-1}(\{e_1,e_{2-},e_3\})\cap\pi_2^{-1}(\{e_1,e_{2-}\})$$
$$=\{(e_1,e_1),(e_{2-},e_{2-}),(e_{2-},e_{2+}),(e_3,e_3)\}\cap\{(e_1,e_1),(e_{2-},e_{2-}),(e_{2+},e_{2-})\}$$
$$\{(e_1,e_1),(e_{2-},e_{2-})\}=U;$$
$$\overline{V}=\pi_1^{-1}(\overline{\pi_1(V)})\cap\pi_2^{-1}(\overline{\pi_2(V)}')=\pi_1^{-1}(\overline{\pi_1(\{(e_{2-},e_{2+})\})})\cap\pi_2^{-1}(\overline{\pi_2(\{(e_{2-},e_{2+})\})}')$$
$$=\pi_1^{-1}(\overline{\{e_{2-}\}})\cap\pi_2^{-1}(\overline{\{e_{2+}\}}')=\pi_1^{-1}(\{e_{2-}\})\cap\pi_2^{-1}(\{e_{2+},e_3\})$$
$$=\{(e_{2-},e_{2-}),(e_{2-},e_{2+})\}\cap\{(e_{2-},e_{2+}),(e_{2+},e_{2+}),(e_3,e_3)\}=\{(e_{2-},e_{2+})\}=V.$$
That is,
\begin{center}
	$\overline{U}=U$ and $\overline{V}=V$, while $\overline{U\cup V}\neq U\cup V$
\end{center}
in $E\times_BE'$ defined as the pullback in $\mathbf{FCLS}$, which is what could not happen in a preorder (since it could not happen in a topological space in general).

\section{Surjective closed and open maps are effective descent morphisms}

Returning to the context of Section 3 and using a result of \cite{[ST1992]}, we easily obtain:

\begin{teo}
	Let $U:\mathbf{C}\to\mathsf{Sets}$ be a faithful functor between categories with pullbacks and coequalizers of equivalence relations that preserves these constructions, and let $\mathcal{P}$ be a class of regular epimorphisms in $\mathbf{C}$ satisfying the following conditions:
	\begin{itemize}
		\item [(a)] $\mathcal{P}$ is pullback stable;
		\item [(b)] if $$\xymatrix{X\ar@<0.5ex>[r]^-f\ar@<-0.5ex>[r]_-g&Y\ar[r]^h&Z}$$ is a coequalizer diagram in $\mathbf{C}$ whose $U$-image is exact, that is, it is a coequalizer diagram that is also a kernel pair diagram, then $f,g\in\mathcal{P}\Rightarrow h\in\mathcal{P}$.
	\end{itemize}
	Then $\mathcal{P}$ is contained in the class of effective descent morphisms in $\mathbf{C}$.
\end{teo}
\begin{proof}
	As follows from (a) and the fact that $\mathcal{P}$ is a class of regular epimorphisms in $\mathbf{C}$, $\mathcal{P}$ is a class of pullback stable regular epimorphisms in $\mathbf{C}$. Note also that, for every descent data $(C,\gamma,\xi)$ over a given $p:E\to B$ in $\mathcal{P}$, we have
	\begin{itemize}
		\item since $U(p)$ being a regular epimorphism is an effective descent morphism in $\mathsf{Sets}$, the $U$-image of the coequalizer diagram $$\xymatrix{E\times_BC\ar@<0.5ex>[r]^-{\xi}\ar@<-0.5ex>[r]_-{\pi_2}&C\ar[r]^q&A}$$
		is exact;
		\item as follows from (a), the morphisms $\xi$ and $\pi_2$ in that diagram belong to $\mathcal{P}$.
	\end{itemize}	
	After that all we need is to apply the categorical counterpart of Corollary 2.8 in \cite{[ST1992]}, as the next sentence (after Corollary 2.8) in \cite{[ST1992]} shows.
\end{proof}

By a \textit{closed map} we mean a morphism $\mathbf{CLS}$ that is closed, or, equivalently, satisfies the equivalent conditions of Proposition 2.5. Similarly, by an \textit{open map} we mean a morphism $\mathbf{CLS}$ that is open, or, equivalently, satisfies the equivalent conditions of Proposition 2.6. In the rest this section we will show that Theorem 6.1  applies to the classes of surjective closed maps and of surjective open maps in $\mathbf{CLS}$.

\begin{prop}
	The class of closed maps is pullback stable. In particular, so is the class of surjective closed maps.
\end{prop}
\begin{proof}
	Consider the pullback for $(p,\alpha)$ with closed $p$. We have to prove that the map $\pi_2:E\times_BA\to A$ is closed. However, this follows from Proposition 2.1 and the fact that we have $$\pi_2(\pi_1^{-1}(E')\cap\pi_2^{-1}(A'))=\alpha^{-1}(p(E'))\cap A'$$ for all $E'\subseteq E$ and $A'\subseteq A$. Indeed, if $E'$ is closed in $E$ and $A'$ is closed in $A$, then $\alpha^{-1}(p(E'))\cap A'$ is closed in $A$ since $p$ is a closed map.
\end{proof}
\begin{prop}
	The class of surjective open maps is pullback stable.
\end{prop}
\begin{proof}
	Consider the pullback for $(p,\alpha)$ with open $p$. We have to prove that the map $\pi_2:E\times_BA\to A$ is open. For $U\subseteq A$, we have
	$$\overline{\pi_2^{-1}(U)}=\pi_1^{-1}(\overline{\pi_1(\pi_2^{-1}(U))})\cap\pi_2^{-1}(\overline{\pi_2(\pi_2^{-1}(U))})$$
	$$=\pi_1^{-1}(\overline{p^{-1}(\alpha(U))})\cap\pi_2^{-1}(\overline{U})=\pi_1^{-1}(p^{-1}(\overline{\alpha(U)}))\cap\pi_2^{-1}(\overline{U})$$
	$$=\pi_2^{-1}(\alpha^{-1}(\overline{\alpha(U)}))\cap\pi_2^{-1}(\overline{U})=\pi_2^{-1}(\alpha^{-1}(\overline{\alpha(U)})\cap\overline{U})$$
	and since
	$$\overline{U}\subseteq\alpha^{-1}(\alpha(\overline{U}))\subseteq\alpha^{-1}(\overline{\alpha(U)}),$$
	this gives $\overline{\pi_2^{-1}(U)}=\pi_2^{-1}(\overline{U})$.
	Therefore $\pi_2$ is open by Proposition 2.6.
\end{proof}
\begin{prop}
	The classes of surjective closed maps and of surjective open maps both satisfy condition (b) of Theorem 6.1 for $U$ being the forgetful functor $\mathbf{CLS}\to\mathsf{Sets}$.
\end{prop}
\begin{proof}
	Consider the diagram of 6.1(b). At the level of underlying sets, the diagram $$\xymatrix{X\ar[d]_f\ar[r]^-g&Y\ar[d]^h\\Y\ar[r]_h&Z}$$ is a pullback, and so, for each subset $Y'$ of $Y$, we have $h^{-1}(h(Y'))=f(g^{-1}(Y'))$. Since $h$ is a regular epimorphism, for closed $f$ this gives:
	\begin{center}
		$Y'$ is closed $\Rightarrow$ $g^{-1}(Y')$ is closed $\Rightarrow$ $f(g^{-1}(Y'))$ is closed\\
		$h^{-1}(h(Y'))$ is closed $\Rightarrow$ $h(Y')$ is closed,
	\end{center}
	and, similarly, for open $f$:
	\begin{center}
		$Y'$ is open $\Rightarrow$ $g^{-1}(Y')$ is open $\Rightarrow$ $f(g^{-1}(Y'))$ is open\\
		$h^{-1}(h(Y'))$ is open $\Rightarrow$ $h(Y')$ is open,
	\end{center}
	as desired.
\end{proof}

From Theorem 6.1 and these three propositions, as promised, we obtain:

\begin{teo}
	Every surjective closed map and every surjective open map in $\mathbf{CLS}$ is an effective descent morphism.\qed
\end{teo}

\section{Final remarks}

\textbf{7.1.} For a morphism $p:E\to B$ in $\mathbf{CLS}$, which is surjective, let us call a subset $Y$ of $E$ \textit{saturated} if it is of the form $Y=p^{-1}(X)$ for some $X\subseteq B$, or, equivalenly, if $Y=p^{-1}(p(Y))$. Consider the following conditions on $p$:
\begin{itemize}
	\item [(a)] $p(Y)$ is closed whenever $Y$ is saturated and closed, or, equivalently (by Corollary 2.3), $p$ is a regular epimorphism in $\mathbf{CLS}$;
	\item [(b)] $p(Y)$ is closed whenever $Y$ is the closure of a saturated subset, or, equivalently (by Proposition 2.10), $p$ is a pullback stable regular epimorphism (=descent morphism) in $\mathbf{CLS}$;
	\item [(c)] $p$ is an effective descent morphism in $\mathbf{CLS}$.
	\item [(d)] $p(Y)$ is closed whenever so is $Y$.
\end{itemize}
We have (d)$\Rightarrow$(c) (Theorem 6.5) and trivial implications (c)$\Rightarrow$(b)$\Rightarrow$(a). It seems that none of the opposite implications holds. In fact it is very easy to construct counterexamples for (a)$\Rightarrow$(b) and, using Theorem 6.5 for (c)$\Rightarrow$(d), but we have no counterexamples for (b)$\Rightarrow$(c).\vspace{1mm}

\textbf{7.2.} For a monad $T$ on the category of sets, consider the forgetful functor $$U:\mathsf{Alg}(T)\to\mathbf{CLS}.$$ The category $\mathsf{Alg}(T)$ is Barr exact and, for a morphism $p$ in it, we have
\begin{center}
	$p$ in an effective descent morphism $\Leftrightarrow$ $p$ is a surjective map,
\end{center}
and the functor $U$ sends all morphisms of $\mathsf{Alg}(T)$ to closed maps; in particular it preserves regular epimorphisms, descent morphisms, and effective descent morphisms. However, it obviously does not preserve kernel pairs of non-injective maps.\vspace{1mm}

\textbf{7.3.} Let $\mathbf{E}$ be one of the following three classes of morphisms in $\mathbf{CLS}$: (i) of closed maps; (ii) of surjective closed maps; (iii) of surjective open maps. As follows from Propositions 6.2 and 6.3 (and simple arguments used in the proof of Proposition 6.4), every effective descent morphism in $\mathbf{CLS}$ is also an effective $\mathbf{E}$-descent morphism. And it is obvious that every descent morphism in $\mathbf{CLS}$ is also an $\mathbf{E}$-descent morphism. However, none of these assertions is true for the class of (all) open maps. Indeed, consider the pullback diagram $$\xymatrix{\{-1,1\}\ar[d]_\beta\ar[r]^-q&\{1\}\ar[d]^\alpha\\\{-2,-1,1,2\}\ar[r]_-p&\{1,2\}}$$ in which:
\begin{itemize}
	\item $\{-2,-1,1,2\}$ has five closed subsets; apart from itself and the empty set they are $\{-2,2\}$, $\{1,2\}$, and \{2\}.
	\item $\{1,2\}$ has three closed subsets; apart from itself and the empty set it is just the set $\{2\}$.
	\item $p$ is defined by $p(k)=|k|$.
	\item $\alpha$ and $\beta$ are the inclusion maps, $q$ is induced by $p$, and the closure space structures on the top are induced by the bottom ones; that is, $$\mathcal{C}_{\{1\}}=\{\emptyset,\{1\}\},\,\,\,\mathcal{C}_{\{-1,1\}}=\{\emptyset,\{1\},\{-1,1\}\}$$ (this makes $\{-1,1\}$ isomorphic to $\{1,2\}$, but that is not relevant for our purposes).
\end{itemize}
It is easy to check that $p$ and $\alpha$ are open maps; furthermore, since $p$ is sujective, it is an effective descent morphism. On the other hand, $\beta$ is not open since $\{-1\}$ is open in $\{-1,1\}$ but not in $\{-2,-1,1,2\}$, and so the pullback functor along $p$ is not even well defined for the class of all open maps.

In spite of all this, a complete characterization of effective $\mathbf{E}$-descent morphisms remains an open question for $\mathbf{E}$ being any of the four classes of morphisms that appear in this subsection. Of course in the `forth case', that is, when $\mathbf{E}$ is the class of open maps, one should suitably reformulate the problem first characterizing those $p:E\to B$ in $\mathbf{CLS}$ for which the above-mentioned pullback functor \textit{is} well defined.


\begin{thebibliography}{99}
	
\bibitem{[CJ2020]} M. M. Clementino, G. Janelidze, {\em Another note on effective descent morphisms of topological spaces and relational algebras}, Topology and its Applications 273, 2020, 106961, 8 pp.

\bibitem{[JS2002]} G. Janelidze, M. Sobral, {\em Finite preorders and topological descent I}, Journal of Pure and Applied Algebra 175(1-3), 2002, 187-205

\bibitem{[JS2020]} G. Janelidze, M. Sobral, Strict monadic topology I: First separation axioms and reflections. Topology Appl. 273 (2020), 106963, 10 pp

\bibitem{[JST2004]} G. Janelidze, M. Sobral, W. Tholen, Beyond Barr exactness: effective descent morphisms, \textit{Categorical Foundations; Special Topics in Order, Topology, Algebra, and Sheaf Theory}, Cambridge University Press, 2004, 359-405

\bibitem{[JT1994]} G. Janelidze, W. Tholen, Facets of Descent I, Applied Categorical Structures 2, 1994, 245-281

\bibitem {[M2011]} Gr. Mirhosseinkhani, On some classes of quotient maps in closure spaces, Int. Math. Forum 6 (2011), no. 21-24, 1155-1161

\bibitem {[ST1992]} M. Sobral, W. Tholen, Effective descent morphisms and effective equivalence relations, Category theory 1991 (Montreal, PQ, 1991), 421–433, CMS Conf. Proc., 13, Amer. Math. Soc., Providence, RI, 1992

\end{thebibliography}
\end{document}